\documentclass{article}
\usepackage{amsfonts}
\usepackage{amsmath}
\usepackage{harvard}

\newtheorem{theorem}{Theorem}

\newtheorem{definition}[theorem]{Definition}

\newtheorem{lemma}[theorem]{Lemma}

\newtheorem{proposition}[theorem]{Proposition}
\newtheorem{remark}[theorem]{Remark}

\newenvironment{proof}[1][Proof]{\noindent\textbf{#1.} }{\ \rule{0.5em}{0.5em}}
\numberwithin{theorem}{section}
\numberwithin{equation}{section}
\input{tcilatex}

\begin{document}

\title{On Cartan Spaces with the $m$-th Root Metric $K(x,p)=\sqrt[m]{%
a^{i_{1}i_{2}...i_{m}}(x)p_{i_{1}}p_{i_{2}}...p_{i_{m}}}$}
\author{Gheorghe Atanasiu and Mircea Neagu \and {\small January 2008;
Revised September 2008 (minor revisions)}}
\date{}
\maketitle

\begin{abstract}
The aim of this paper is to expose some geometrical properties of the
locally Minkowski-Cartan space with the Berwald-Mo\'{o}r metric of momenta $%
L(p)=\sqrt[n]{p_{1}p_{2}...p_{n}}$. This space is regarded as a particular
case of the $m$-th root Cartan space. Thus, Section 2 studies the $v-$%
covariant derivation components of the $m$-th root Cartan space. Section 3
computes the $v-$curvature d-tensor $S^{hijk}$ of the $m$-th root Cartan
space and studies conditions for $S3-$likeness. Section 4 computes the $T-$%
tensor $T^{hijk}$ of the $m$-th root Cartan space. Section 5 particularizes
the preceding geometrical results for the Berwald-Mo\'{o}r metric of momenta.
\end{abstract}



\section{Introduction}

\hspace{4mm} Owing to the studies of E. Cartan, R. Miron [6], [7], Gh.
Atanasiu [2] and many others, the geometry of Cartan spaces is today an
important chapter of differential geometry, regarded as a particular case of
the Hamilton geometry. By the Legendre duality of the Cartan spaces with the
Finsler spaces studied by R. Miron, D. Hrimiuc, H. Shimada and S. V. Sab\u{a}%
u [8], it was shown that the theory of Cartan spaces has the same symmetry
like the Finsler geometry, giving in this way a geometrical framework for
the Hamiltonian theory of Mechanics or Physical fields. In a such
geometrical context we recall that a \textit{Cartan space} is a pair $%
\mathcal{C}^{n}=\left( M^{n},K(x,p)\right) $ such that the following axioms
hold good:

\begin{enumerate}
\item $K$ is a real positive function on the cotangent bundle $T^{\ast }M$,
differentiable on $T^{\ast }M\backslash \{0\}$ and continuous on the null
section of the canonical projection 
\begin{equation*}
\pi ^{\ast }:T^{\ast }M\rightarrow M;
\end{equation*}

\item $K$ is positively $1$-homogenous with respect to the momenta $p_{i}$;

\item The Hessian of $K^{2}$, with the elements%
\begin{equation*}
g^{ij}(x,p)=\frac{1}{2}\frac{\partial ^{2}K^{2}}{\partial p_{i}\partial p_{j}%
},
\end{equation*}%
is positive-defined on $T^{\ast }M\backslash \{0\}$.
\end{enumerate}

On the other hand, in the last two decades, physical studies due to G. S.
Asanov [1], D. G. Pavlov [9], [10] and their co-workers emphasize the
important role played by the Berwald-Mo\'{o}r metric $L:TM\rightarrow 
\mathbb{R}$,%
\begin{equation*}
L(y)=\left( y^{1}y^{2}...y^{n}\right) ^{\frac{1}{n}},
\end{equation*}%
in the theory of space-time structure and gravitation as well as in unified
gauge field theories.

For such geometrical-physics reasons, following the geometrical ideas
exposed by M. Matsumoto and H. Shimada in [4], [5] and [11] or by ourselves
in [3], in this paper we investigate some geometrical properties of the $m$%
\textit{-th root Cartan space} which is a natural generalization of the
locally Minkowski-Cartan space with the Berwald-Mo\'{o}r metric of momenta.

\section{The $m$-th root metric and $v-$derivation components}

\hspace{4mm} Let $\mathcal{C}^{n}=\left( M^{n},K(x,p)\right) $, $n\geq 4$,
be an $n$-dimensional Cartan space with the metric%
\begin{equation}
K(x,p)=\sqrt[m]{a^{i_{1}i_{2}...i_{m}}(x)p_{i_{1}}p_{i_{2}}...p_{i_{m}}},
\label{m-root-metric}
\end{equation}%
where $a^{i_{1}i_{2}...i_{m}}(x)$, depending on the position alone, is
symmetric in all the indices $i_{1},i_{2},...,i_{m}$ and $m\geq 3$.

\begin{definition}
The Cartan space with the metric (\ref{m-root-metric}) is called the $m$%
\textbf{-th root Cartan space}.
\end{definition}

Let us consider the following notations:%
\begin{equation}
\begin{array}{l}
a^{i}=\left[ a^{ii_{2}i_{3}...i_{m}}(x)p_{i_{2}}p_{i_{3}}...p_{i_{m}}\right]
/K^{m-1},\medskip \\ 
a^{ij}=\left[ a^{iji_{3}i_{4}...i_{m}}(x)p_{i_{3}}p_{i_{4}}...p_{i_{m}}%
\right] /K^{m-2},\medskip \\ 
a^{ijk}=\left[ a^{ijki_{4}i_{5}...i_{m}}(x)p_{i_{4}}p_{i_{5}}...p_{i_{m}}%
\right] /K^{m-3}.%
\end{array}
\label{notatii-auri}
\end{equation}

The normalized supporting element%
\begin{equation*}
l^{i}=\dot{\partial}^{i}K,\text{ where }\dot{\partial}^{i}=\frac{\partial }{%
\partial p_{i}},
\end{equation*}%
the fundamental metrical d-tensor%
\begin{equation*}
g^{ij}=\frac{1}{2}\dot{\partial}^{i}\dot{\partial}^{j}K^{2}
\end{equation*}%
and the angular metrical d-tensor 
\begin{equation*}
h^{ij}=K\dot{\partial}^{i}\dot{\partial}^{j}K
\end{equation*}%
are given by%
\begin{equation}
\begin{array}{l}
l^{i}=a^{i},\medskip \\ 
g^{ij}=(m-1)a^{ij}-(m-2)a^{i}a^{j},\medskip \\ 
h^{ij}=(m-1)(a^{ij}-a^{i}a^{j}).%
\end{array}
\label{lgh-formulas}
\end{equation}

\begin{remark}
From the positively $1$-homogeneity of the $m$-th root Cartan metrical
function (\ref{m-root-metric}) it follows that we have%
\begin{equation*}
K^{2}(x,p)=g^{ij}(x,p)p_{i}p_{j}=a^{ij}(x,p)p_{i}p_{j}.
\end{equation*}
\end{remark}

Let us suppose now that the d-tensor $a^{ij}$ is regular, that is there
exists the inverse matrix $(a^{ij})^{-1}=(a_{ij})$. Obviously, we have%
\begin{equation*}
a_{i}\cdot a^{i}=1,
\end{equation*}%
where%
\begin{equation*}
a_{i}=a_{is}a^{s}=\frac{p_{i}}{K}.
\end{equation*}%
Under these assumptions, we obtain the inverse components $g_{ij}(x,p)$ of
the fundamental metrical d-tensor $g^{ij}(x,p)$, which are given by%
\begin{equation}
g_{ij}=\frac{1}{m-1}a_{ij}+\frac{m-2}{m-1}a_{i}a_{j}.  \label{inverse(g)}
\end{equation}

The relations (\ref{notatii-auri}) and (\ref{lgh-formulas}) imply that the
components of the $v-$torsion d-tensor%
\begin{equation*}
C^{ijk}=-\frac{1}{2}\dot{\partial}^{k}g^{ij}
\end{equation*}%
are given in the form%
\begin{equation}
C^{ijk}=-\frac{(m-1)(m-2)}{2K}\left(
a^{ijk}-a^{ij}a^{k}-a^{jk}a^{i}-a^{ki}a^{j}+2a^{i}a^{j}a^{k}\right) .
\label{v-torsion(C^ijk)}
\end{equation}%
Consequently, using the relations (\ref{inverse(g)}) and (\ref%
{v-torsion(C^ijk)}), together with the formula 
\begin{equation*}
a_{s}a^{sjk}=a^{jk},
\end{equation*}%
we find the components of the $v-$derivation%
\begin{equation*}
C_{i}^{jk}=g_{is}C^{sjk}
\end{equation*}%
in the following form:%
\begin{equation}
C_{i}^{jk}=-\frac{(m-2)}{2K}\left[ a_{i}^{jk}-\left( \delta
_{i}^{j}a^{k}+\delta _{i}^{k}a^{j}\right) +a_{i}(2a^{j}a^{k}-a^{jk})\right] ,
\label{v-derivation(C_i^jk)}
\end{equation}%
where%
\begin{equation*}
a_{i}^{jk}=a_{is}a^{sjk}.
\end{equation*}%
From (\ref{v-derivation(C_i^jk)}) we easily find the following geometrical
result:

\begin{proposition}
The \textbf{torsion covector}%
\begin{equation*}
C^{i}=C_{r}^{ir}
\end{equation*}%
is given by the formula%
\begin{equation*}
C^{i}=-\frac{(m-2)}{2K}\left( a_{r}^{ir}-na^{i}\right) ,
\end{equation*}%
where $n=\dim M$.
\end{proposition}

\section{The $v-$curvature d-tensor $S^{hijk}$}

\hspace{4mm} Taking into account the relations (\ref{v-torsion(C^ijk)}) and (%
\ref{v-derivation(C_i^jk)}), by calculation, we obtain

\begin{theorem}
The $v-$\textbf{curvature} d-tensor%
\begin{equation*}
S^{hijk}=C_{r}^{ij}C^{rhk}-C_{r}^{ik}C^{rhj}
\end{equation*}%
can be written in the form%
\begin{equation*}
S^{hijk}=\frac{(m-1)(m-2)^{2}}{4K^{2}}\mathcal{A}_{\{j,k\}}%
\{a_{r}^{ij}a^{rhk}-a^{ij}(a^{hk}-a^{h}a^{k})+a^{i}a^{j}a^{hk}\},
\end{equation*}%
where $\mathcal{A}_{\{j,k\}}$ means an alternate sum.
\end{theorem}

\begin{remark}
Using the relations (\ref{lgh-formulas}), we underline that the $v-$%
curvature d-tensor $S^{hijk}$ can be written as%
\begin{equation}
K^{2}S^{hijk}=\frac{(m-2)^{2}}{4}\left[ \left(
h^{hj}h^{ik}-h^{hk}h^{ij}\right) /(m-1)+(m-1)U^{hijk}\right] ,
\label{formula((S)si(h))}
\end{equation}%
where%
\begin{equation}
U^{hijk}=a_{r}^{ij}a^{rhk}-a_{r}^{ik}a^{rhj}.  \label{tensor(U^hijk)}
\end{equation}
\end{remark}

In the sequel, let us recall the following important geometrical concept [4]:

\begin{definition}
A Cartan space $\mathcal{C}^{n}=\left( M^{n},K(x,p)\right) $, $n\geq 4$, is
called $S3$\textbf{-like} if there exists a positively $0$-homogenous scalar
function $S=S(x,p)$ such that the $v-$curvature d-tensor $S^{hijk}$ to have
the form%
\begin{equation}
K^{2}S^{hijk}=S\left\{ h^{hj}h^{ik}-h^{hk}h^{ij}\right\} .  \label{S3-like}
\end{equation}
\end{definition}

Let $\mathcal{C}^{n}=\left( M^{n},K(x,p)\right) $, $n\geq 4$, be the $m$-th
root Cartan space. As an immediate consequence of the above definition we
have the following important result:

\begin{theorem}
The $m$-th root Cartan space $\mathcal{C}^{n}$ is an $S3$-like Cartan space
if and only if the d-tensor $U^{hijk}$ is of the form%
\begin{equation}
U^{hijk}=\lambda \left\{ h^{hj}h^{ik}-h^{hk}h^{ij}\right\} ,
\label{S3-like-condition}
\end{equation}%
where $\lambda =\lambda (x,p)$ is a positively $0$-homogenous scalar
function.
\end{theorem}

\begin{proof}
Taking into account the formula (\ref{formula((S)si(h))}) and the condition (%
\ref{S3-like-condition}), we find the scalar function (see (\ref{S3-like}))%
\begin{equation}
S=\frac{(m-2)^{2}}{4}\left[ (m-1)\lambda +\frac{1}{m-1}\right] .
\label{relation((S)si(lambda))}
\end{equation}
\end{proof}

\section{The $T-$tensor $T^{hijk}$}

\hspace{5mm} Let $N=(N_{ij})$ be the canonical nonlinear connection of the $%
m $-th root Cartan space with the metric (\ref{m-root-metric}), whose local
coefficients are given by [8]%
\begin{equation*}
N_{ij}=\gamma _{ij}^{0}-\frac{1}{2}\gamma _{h0}^{0}\dot{\partial}^{h}g_{ij},
\end{equation*}%
where%
\begin{equation*}
\begin{array}{l}
\partial _{k}=\dfrac{\partial }{\partial x^{k}},\medskip \\ 
\gamma _{jk}^{i}=\dfrac{g^{ir}}{2}(\partial _{k}g_{rj}+\partial
_{j}g_{rk}-\partial _{r}g_{jk}),\medskip \\ 
\gamma _{ij}^{0}=-\gamma _{ij}^{s}p_{s},\medskip \\ 
\gamma _{h0}^{0}=\gamma _{hr}^{l}g^{rs}p_{l}p_{s}.%
\end{array}%
\end{equation*}

Let 
\begin{equation*}
C\Gamma (N)=(H_{jk}^{i},C_{i}^{jk})
\end{equation*}%
be the Cartan canonical connection of the $m$-th root Cartan space with the
metric (\ref{m-root-metric}). The local components of the Cartan canonical
connection $C\Gamma (N)$ have the expressions [8]%
\begin{equation*}
\begin{array}{l}
H_{jk}^{i}=\dfrac{g^{ir}}{2}(\delta _{k}g_{rj}+\delta _{j}g_{rk}-\delta
_{r}g_{jk}),\medskip \\ 
C_{i}^{jk}=g_{is}C^{sjk}=-\dfrac{g_{is}}{2}\dot{\partial}^{k}g^{js},%
\end{array}%
\end{equation*}%
where%
\begin{equation*}
\delta _{j}=\partial _{j}+N_{js}\dot{\partial}^{s}.
\end{equation*}

In the sequel, let us compute the $T-$tensor $T^{hijk}$ of the $m$-th root
Cartan space, which is defined as [11]%
\begin{equation*}
T^{hijk}\overset{def}{=}%
KC^{hij}|^{k}+l^{h}C^{ijk}+l^{i}C^{jkh}+l^{j}C^{khi}+l^{k}C^{hij},
\end{equation*}%
where " $|^{k}$ " denotes the local $v-$covariant derivation with respect to 
$C\Gamma (N)$, that is we have%
\begin{equation*}
C^{hij}|^{k}=\dot{\partial}%
^{k}C^{hij}+C^{rij}C_{r}^{hk}+C^{hrj}C_{r}^{ik}+C^{hir}C_{r}^{jk}.
\end{equation*}

Using the definition of the local $v-$covariant derivation [8], together
with the relations (\ref{v-derivation(C_i^jk)}) and (\ref{lgh-formulas}), by
direct computations, we find the relations:%
\begin{equation}
\begin{array}{l}
K|^{k}=a^{k}=l^{k},\medskip \\ 
a^{i}|^{k}=\dfrac{(m-1)}{K}(a^{ik}-a^{i}a^{k})=\dfrac{h^{ik}}{K},\medskip \\ 
a^{ij}|^{k}=\dfrac{(m-2)}{K}(a^{ik}a^{j}+a^{jk}a^{i}-2a^{i}a^{j}a^{k})=%
\dfrac{(m-2)}{(m-1)K}(h^{ik}l^{j}+h^{jk}l^{i}).%
\end{array}
\label{v-derivari (auri)}
\end{equation}

Suppose that we have $m\geq 4$. Then, the notation%
\begin{equation*}
a^{hijk}=\left[ a^{hijki_{5}i_{6}...i_{m}}(x)p_{i_{5}}p_{i_{6}}...p_{i_{m}}%
\right] /K^{m-4}
\end{equation*}
is very useful. In this context, we can give the next geometrical results:

\begin{lemma}
\label{Lema} The $v-$covariant derivation of the tensor $a^{hij}$ is given
by the following formula:%
\begin{eqnarray}
a^{hij}|^{k} &=&\frac{(m-3)}{K}a^{hijk}+\frac{m}{2K}a^{hij}a^{k}-\frac{(m-2)%
}{2K}\cdot \left\{ a_{r}^{hk}a^{rij}+\right.  \label{v-derivation(a^hij)} \\
&&+a_{r}^{ik}a^{rhj}+a_{r}^{jk}a^{rhi}-a^{kij}a^{h}-a^{hkj}a^{i}-a^{hik}a^{j}-
\notag \\
&&-a^{ij}a^{hk}-a^{hj}a^{ik}-a^{hi}a^{jk}+2\left( a^{ij}a^{h}a^{k}+\right. 
\notag \\
&&\left. \left. +a^{hj}a^{i}a^{k}+a^{hi}a^{j}a^{k}\right) \right\} .  \notag
\end{eqnarray}
\end{lemma}

\begin{proof}
Note that, by a direct computation, we obtain the relation%
\begin{equation}
\frac{\partial a^{hij}}{\partial p_{k}}=\frac{(m-3)}{K}\left(
a^{hijk}-a^{hij}a^{k}\right) .  \label{deriv-of-a^hij-wrt-to-p_k}
\end{equation}

Finally, using the definition of the local $v-$covariant derivation,
together with the formulas (\ref{deriv-of-a^hij-wrt-to-p_k}) and (\ref%
{v-derivation(C_i^jk)}), we find the equality (\ref{v-derivation(a^hij)}).
\end{proof}

\begin{theorem}
The $T-$tensor $T^{hijk}$ of the $m$-th root Cartan space is given by the
expression%
\begin{eqnarray}
T^{hijk} &=&-\frac{(m-1)(m-2)(m-3)}{2K}a^{hijk}+\frac{(m-1)(m-2)^{2}}{4K}%
\cdot  \label{T^hijk(m-root)} \\
&&\cdot \left( a_{r}^{hk}a^{rij}+a_{r}^{ik}a^{rhj}+a_{r}^{jk}a^{rhi}\right) -%
\frac{m(m-1)(m-2)}{4K}\cdot  \notag \\
&&\cdot \left(
a^{hij}a^{k}+a^{hjk}a^{i}+a^{ijk}a^{h}+a^{hik}a^{j}-a^{ij}a^{hk}-a^{hj}a^{ik}-a^{ih}a^{jk}\right) .
\notag
\end{eqnarray}
\end{theorem}

\begin{proof}
It is obvious that we have the equality%
\begin{eqnarray*}
T^{hijk} &=&(KC^{hij})|^{k}+l^{h}C^{ijk}+l^{i}C^{jkh}+l^{j}C^{khi}= \\
&=&(KC^{hij})|^{k}+a^{h}C^{ijk}+a^{i}C^{jkh}+a^{j}C^{khi}.
\end{eqnarray*}

Consequently, differentiating $v-$covariantly the relation (\ref%
{v-torsion(C^ijk)}) multiplied by $K$ and using the formulas (\ref%
{v-derivari (auri)}), together with the Lemma \ref{Lema}, by laborious
computations, it follows the required result.
\end{proof}

\section{The particular case of Berwald-Mo\'{o}r metric of momenta}

\hspace{5mm} Let us consider now the particular case when $m=n\geq 4$ and%
\begin{equation*}
a^{i_{1}i_{2}...i_{n}}(x)=\left\{ 
\begin{array}{ll}
1/n!, & i_{1}\neq i_{2}\neq ...\neq i_{n}\medskip \\ 
0, & \text{otherwise.}%
\end{array}%
\right.
\end{equation*}%
In this special case, the $m$-th root metric (\ref{m-root-metric}) becomes
the \textbf{Berwald-Mo\'{o}r metric of momenta} [8]%
\begin{equation}
K(p)=\sqrt[n]{p_{1}p_{2}...p_{n}}.  \label{Berwald-moor metric}
\end{equation}

By direct computations, we deduce that the $n$-dimensional locally
Minkowski-Cartan space $\mathcal{C}^{n}=\left( M^{n},K(p)\right) $ endowed
with the Berwald-Mo\'{o}r metric of momenta (\ref{Berwald-moor metric}) is
characterized by the following geometrical entities and relations denoted by
(E-R):%
\begin{equation*}
\begin{array}{ccc}
a^{i}=\dfrac{K}{n}\cdot \dfrac{1}{p_{i}}, & a_{i}=\dfrac{p_{i}}{K}, & 
a_{i}\cdot a^{i}=\dfrac{1}{n}\text{ (no sum by }i\text{),}%
\end{array}%
\end{equation*}%
\begin{equation*}
a^{ij}=\left\{ 
\begin{array}{ll}
\dfrac{n}{n-1}\cdot a^{i}a^{j}, & i\neq j\medskip \\ 
0, & i=j%
\end{array}%
\right. ,
\end{equation*}%
\begin{equation*}
a_{ij}=\left\{ 
\begin{array}{ll}
n\cdot a_{i}a_{j}, & i\neq j\medskip \\ 
-n(n-2)\cdot (a_{i})^{2}, & i=j,%
\end{array}%
\right.
\end{equation*}%
\begin{equation*}
a^{ijk}=\left\{ 
\begin{array}{ll}
\dfrac{n^{2}}{(n-1)(n-2)}\cdot a^{i}a^{j}a^{k}, & i\neq j\neq k\medskip \\ 
0, & \text{otherwise}%
\end{array}%
\right.
\end{equation*}%
and%
\begin{equation*}
a^{hijk}=\left\{ 
\begin{array}{ll}
\dfrac{n^{3}}{(n-1)(n-2)(n-3)}\cdot a^{h}a^{i}a^{j}a^{k}, & h\neq i\neq
j\neq k\medskip \\ 
0, & \text{otherwise.}%
\end{array}%
\right.
\end{equation*}%
Moreover, the equalities (E-R) imply that the components $a_{i}^{jk}$ are
given by the formulas:%
\begin{equation}
\begin{array}{ll}
a_{i}^{jk}=-\dfrac{n^{2}}{(n-1)(n-2)}\cdot a_{i}a^{j}a^{k}, & i\neq j\neq
k\medskip \\ 
a_{i}^{ik}=a_{i}^{ki}=\dfrac{n}{n-1}\cdot a^{k}, & i\neq k\text{ (no sum by }%
i\text{)}\medskip \\ 
a_{i}^{kk}=0, & \forall \text{ }i=\overline{1,n},\text{ (no sum by }k\text{).%
}%
\end{array}
\label{formulas(a_i^jk)}
\end{equation}

In this context, we obtain the following important geometrical result:

\begin{theorem}
The locally Minkowski-Cartan space $\mathcal{C}^{n}=\left( M^{n},K(p)\right) 
$, $n\geq 4$, endowed with the Berwald-Mo\'{o}r metric of momenta (\ref%
{Berwald-moor metric}) is characterized by the following geometrical
properties:

\begin{enumerate}
\item The torsion covector $C^{i}$ vanish;

\item $S3$-likeness with the scalar function $S=-1$;

\item The $T-$tensor $T^{hijk}$ vanish.
\end{enumerate}
\end{theorem}

\begin{proof}
1. It is easy to see that we have%
\begin{equation*}
\sum_{r}a_{r}^{ir}=\sum_{r,s}a_{rs}a^{sir}=n\sum_{r,s}a_{r}a_{s}a^{sir}=n%
\sum_{r}a_{r}a^{ir}=na^{i}.
\end{equation*}

2. It is obvious that we have 
\begin{equation*}
h^{ij}=\left\{ 
\begin{array}{ll}
a^{i}a^{j}, & i\neq j\medskip \\ 
-(n-1)\cdot (a_{i})^{2}, & i=j.%
\end{array}%
\right.
\end{equation*}%
Consequently, by computations, we obtain%
\begin{equation*}
U^{hijk}=-\frac{n^{2}}{(n-1)^{2}(n-2)^{2}}\left\{
h^{hj}h^{ik}-h^{hk}h^{ij}\right\} ,
\end{equation*}%
where $U^{hijk}$ is given by the relation (\ref{tensor(U^hijk)}). It follows
what we were looking for (see the equalities (\ref{S3-like-condition}) and (%
\ref{relation((S)si(lambda))})).

3. Using the relation (\ref{T^hijk(m-root)}) and the formulas (E-R) and (\ref%
{formulas(a_i^jk)}), by laborious computations, we deduce that $T^{hijk}=0$.
\end{proof}

\textbf{Acknowledgements. }Research supported by Romanian Academy Grant No.
5 / 5.02.2008.

\begin{center}
Gheorghe ATANASIU and Mircea NEAGU

University "Transilvania" of Bra\c{s}ov

Faculty of Mathematics and Informatics

Department of Algebra, Geometry and Differential Equations

Bd. Eroilor, Nr. 29, Bra\c{s}ov, BV 500019, Romania.

gh\_atanasiu@yahoo.com

mirceaneagu73@yahoo.com
\end{center}


\begin{thebibliography}{99}
\bibitem{Asanov[1]} G.S. Asanov, \textit{Finslerian Extension of General
Relativity}, Reidel, Dordrecht, 1984.

\bibitem{Atanasiu[2]} Gh. Atanasiu, \textit{The Invariant Expression of
Hamilton Geometry}, Tensor N. S. \textbf{47}, \textbf{3} (1988), 225-234.

\bibitem{At-Bal-Neagu[3]} Gh. Atanasiu, V. Balan, M. Neagu, \textit{The
Pavlov's }$4$\textit{-Polyform of Momenta }$K(p)=\sqrt[4]{%
p_{1}p_{2}p_{3}p_{4}}$\textit{\ and Its Applications in Hamilton Geometry},
Hypercomplex Numbers in Geometry and Physics\textit{\ }\textbf{2}, \textbf{4}
(2005), 134-139.

\bibitem{Matsumoto[4]} M. Matsumoto, \textit{On Finsler Spaces with
Curvature Tensors of Some Special Forms}, Tensor N. S. \textbf{22} (1971),
201-204.

\bibitem{Mats-Shimada[5]} M. Matsumoto, H. Shimada, \textit{On Finsler
Spaces with }$1$\textit{-Form Metric. II. Berwald-Mo\'{o}r's Metric }$%
L=\left( y^{1}y^{2}...y^{n}\right) ^{1/n}$, Tensor N. S. \textbf{32} (1978),
275-278.

\bibitem{Miron-Hamilton[6]} R. Miron, \textit{Hamilton Geometry}, An. \c{S}%
t. "Al. I. Cuza" Univ. Ia\c{s}i \textbf{35} (1989), 33-67.

\bibitem{Miron-Cartan[7]} R. Miron, \textit{The Geometry of Cartan Spaces},
Prog. in Math. \textbf{22} (1988), 1-38.

\bibitem{Miron-Hr-Shimada[8]} R. Miron, D. Hrimiuc, H. Shimada, S.V. Sab\u{a}%
u, \textit{The Geometry of Hamilton and Lagrange Spaces}, Kluwer Academic
Publishers, 2001.

\bibitem{Pavlov(1)[9]} D.G. Pavlov, \textit{Four-Dimensional Time},
Hypercomplex Numbers in Geometry and Physics \textbf{1}, \textbf{1} (2004),
31-39.

\bibitem{Pavlov(2)[10]} D.G. Pavlov, \textit{Generalization of Scalar
Product Axioms}, Hypercomplex Numbers in Geometry and Physics \textbf{1}, 
\textbf{1} (2004), 5-18.

\bibitem{Shimada[11]} H. Shimada, \textit{On Finsler Spaces with the Metric }%
$L(x,y)=\sqrt[m]{a_{i_{1}i_{2}...i_{m}}(x)y^{i_{1}}y^{i_{2}}...y^{i_{m}}}$,
Tensor N. S. \textbf{33} (1979), 365-372.
\end{thebibliography}
\end{document}